\documentclass[a4paper,11pt]{amsart}
  \usepackage{amsmath}
  \usepackage{url,graphicx}
  \usepackage{amssymb, color, pstricks-add, manfnt}
  \usepackage{amsmath, amsthm,  amsfonts}

  \theoremstyle{plain}
  \newtheorem{Theorem}{Theorem}[section]
  \newtheorem{Lemma}{Lemma}[section]

  \numberwithin{equation}{section}
  \numberwithin{figure}{section}

\renewcommand{\baselinestretch}{1.00}
\begin{document}

\title{Monge-Amp\`ere type equations with Neumann boundary conditions on Riemannian manifolds}

\author{Xi Guo}
\address{Faculty of Mathematics and Statistics, Hubei Key Laboratory of Applied Mathematics, Hubei University,  Wuhan 430062, P.R. China}
\email{guoxi@hubu.edu.cn}
\author{Jing Mao}
\address{Faculty of Mathematics and Statistics, Hubei Key Laboratory of Applied Mathematics, Hubei University,  Wuhan 430062, P.R. China}
\email{jiner120@163.com,jiner120@tom.com}
\author{Ni Xiang}
\address{Faculty of Mathematics and Statistics, Hubei Key Laboratory of Applied Mathematics, Hubei University,  Wuhan 430062, P.R. China}
\email{nixiang@hubu.edu.cn}


\date{\today}


\abstract {In this paper, we consider the global regularity for Monge-Amp\`ere type equations with the Neumann boundary conditions on Riemannian manifolds. It is known that the
classical solvability of the Neumann boundary value problem is
obtained under some necessary assumptions. Our main result extends the main theorem from the case of Euclidean space $R^n$ in \cite{JTX} to Riemannian manifolds.}
\endabstract

\maketitle


\baselineskip=12.8pt
\parskip=3pt
\renewcommand{\baselinestretch}{1.38}

\section{Introduction}\label{Section 1}

\vskip10pt
The main purpose of this paper is to study the Neumann boundary value problem for the Monge-Amp\`ere type equations on Riemannian manifolds. Let $(M^n, g)$ be an $n\geq 2$ dimensional smooth Riemannian manifold. $S_2M^n$ is the bundle of symmetric (0,2) tensor on
$M^n$, and $\Omega\subset M^n$ is a compact domain with smooth boundary $\partial\Omega$.
We consider the equation
\begin{equation}\label{1.1}
\det[(\nabla^2u-A(x,u,\nabla u))g^{-1}]=B(x,u,\nabla u) \ \  in \ \ \Omega,
\end{equation}
together with the Neumann boundary condition
\begin{equation}\label{1.2}
\nabla_\nu u = \varphi(x,u) \quad \mbox{on } \  \partial\Omega,
\end{equation}
where $A: \bar{\Omega}\times\mathbb{R}\times T_xM^n\rightarrow S_2M^n$,
$B\geq0$ is $C^\infty$ with respect to $(x,z,p)\in\bar{\Omega}\times\mathbb{R}\times T_xM^n$. Here $T_x M^n$ denotes the tangent space at $x\in M^n$, and $\nu$ is the unit inner normal vector field on
$\partial\Omega$. As customary $\nabla u$ and $\nabla^2u$ denote respectively the gradient vector and Hessian matrix of second derivatives of $u$. A solution $u$ is elliptic, (degenerate elliptic), if the matrix ${\nabla^2u-A(x,u,\nabla u)}$ is positive, (non negative) definite.

There has been considerable research activity in recent years devoted to fully nonlinear elliptic, second order partial differential equations of the form (\ref{1.1}), which arise in applications, notably in optimal transportation \cite{MTW2005, TruWang2009} and also in reflector and refractor shape design problems \cite{GH2009, Wang1996, Wang2006}.

The Yamabe problem on manifold with boundary was studied by Escobar \cite{ES}, he showed that almost every compact Riemannian manifold was equivalent to the constant scalar curvature manifold, whose boundary was minimal. In fact the problem can be reduced to solve the semilinear elliptic critical Sobolev exponent equation with Neumann boundary condition.
The Neumann boundary problem of linear and quasilinear elliptic equation was widely studied for a long time, readers can see the recent book written by Lieberman\cite{Lie2013}. The motivation of studying the Neumann boundary value problem for Monge-Amp\`ere type equations comes from its application in conformal geometry. Such a prescribed mean curvature problem in conformal geometry was first proposed in \cite{LiLi2006}.

There are also many known results about the fully nonlinear elliptic equations on Riemannian manifolds. For example, Guan and Li in \cite{GL1996} extended the well-known result for Monge-Amp\`ere equation with Dirichlet boundary condition in $\mathbb{R}^n$. For more results, we refer readers to the  articles \cite{ AZ,Guan2014,GS} and references therein.

In this paper, we shall derive a priori second-order estimates for solutions of the Neumann boundary value problem \eqref{1.1}-\eqref{1.2} on Riemannian manifolds. It is well known that these estimates yield regularity and existence results. For this aim, the regularity of solutions depends on the behaviour of the matrix $A$ with respect to the $p$ variables. We call the matrix $A$ regular if $A$ is co-dimension one convex with respect to $p$, in the sense that
\begin{equation}\label{regular}
\nabla_{p_kp_l}A_{ij}(x,z,p)\xi_i\xi_j\eta_k\eta_l \geq 0,
\end{equation}
for all $(x,z,p)\in\Omega\times \mathbb{R}\times T_xM^n$,
$\xi,\eta\in T_xM^n$, $\xi\perp\eta$.
 If \eqref{regular} is replaced by
\begin{equation}\label{regular3}
\nabla_{p_kp_l}A_{ij}(x,z,p)\xi_i\xi_j\eta_k\eta_l \geq c_0|\xi|^2|\eta|^2,
\end{equation}
for some $c_0>0$, then the matrix $A$ is called strictly regular. Conditions (\ref{regular}) and (\ref{regular3}) were introduced in \cite{MTW2005, Tru2006} and called A3w, A3 respectively. Loeper in \cite{Loeper2009} showed that the condition A3w was indeed the necessary and sufficient condition for regularity. One can not expect regularity without this condition. A case of special interest for applications is the generalization of Brenier's cost to Riemannian manifolds. Existence and uniqueness of optimal maps in that case was established by McCann \cite{Mc}.

As with \cite{JT2014,Tru2014}, we also need monotone assumptions about $A$, $B$ and $\varphi$. The matrix $A$ is non-decreasing (strictly increasing) with respect to $z$, if
\begin{equation}\label{monotone A}
\nabla_zA_{ij}(x,z,p)\xi_i\xi_j\ge 0 (>0),
\end{equation}
for all $(x,z,p)\in\Omega\times \mathbb{R}\times T_xM^n $, $\xi\in \mathbb{R}^n$.
The inhomogeneous term $B$ is non-decreasing (strictly increasing) with respect to $z$, if
\begin{equation}\label{monotone B}
\nabla_zB(x,z,p)\ge 0 (>0),
\end{equation}
for all $(x,z,p)\in\Omega\times \mathbb{R}\times T_x M^n$. The function $\varphi$ defined on the boundary is called non-decreasing (strictly increasing) with respect to $z$, if
\begin{equation}\label{monotone varphi}
\nabla_z\varphi(x,z)\ge 0 (>0),
\end{equation}
for all $(x,z)\in\partial\Omega\times \mathbb{R}$.

As well, we need to assume a kind of global barrier condition called the uniformly $A$-convexity in \cite{Tru2006} for the domain $\Omega$, namely that there exists a defining function $\phi\in C^2(\bar \Omega)$,
satisfying $\phi=0$ on $\partial \Omega$, $\nabla\phi \neq 0$ on
$\partial \Omega$ and $\phi<0$ in $\Omega$, together with the
inequality
\begin{equation}\label{domain A convexity}
\nabla_{ij}\phi - \nabla_{p_k}A_{ij}(x,u,\nabla u)\nabla_k\phi \ge \delta_0 I,
\end{equation}
in a neighbourhood $\mathcal{N}$ of $\partial \Omega$, where $\delta_0$ is a positive constant, $I$
denotes the identity matrix. The inequality \eqref{domain A convexity} is
trivially satisfied in the standard Monge-Amp\`ere case which can be easily seen by taking $\phi(x)=|x|^2$ in $R^n$. For the Monge-Amp\`ere equations on manifolds, (\ref{domain A convexity}) is a natural condtion for existence of global smooth solutions, called existence of a geodesic convex function on $\Omega$ by Hong in \cite{Hong1992}. By virtue of the uniformly $A$-convexity of the domain
$\Omega$, for example,
\begin{equation}
\phi = -a d + b d^2,
\end{equation}
where $a$ and $b$ are positive constants and $d(x)\triangleq dist
(x,\partial \Omega)$ denotes the distance function for $\Omega$, see
\cite{GTbook,TruWang2009} for reference.

In order to achieve the second order derivative estimate under the necessary natural condition \eqref{regular}, we need to assume the existence of a supersolution to the problem \eqref{1.1}-\eqref{1.2} satisfying
\begin{eqnarray}
&\det [(\nabla^2\bar u-A(x,\bar u,\nabla\bar u))g^{-1}]  \le  B(x,\bar u,\nabla\bar u)   & \quad \mbox{in } \ \Omega,\label{super equation}\\
&              \nabla_\nu \bar u  =  \varphi(x,\bar u) & \quad \mbox{on } \  \partial\Omega, \label{super boundary}
\end{eqnarray}
and a subsolution to the problem \eqref{1.1}-\eqref{1.2} satisfying
\begin{eqnarray}
&\det [(\nabla^2\underline{u} -A(x,\underline{u} ,\nabla\underline{u} ))g^{-1}]  \ge  B(x,\underline{u} ,\nabla\underline{u} )   & \quad \mbox{in } \ \Omega,\label{sub equation}\\
&              \nabla_\nu \underline{u}   =  \varphi(x,\underline{u}) & \quad \mbox{on } \  \partial\Omega. \label{sub boundary}
\end{eqnarray}

We now formulate the main results of this paper. First, the global second derivative estimate can be obtained as follows.
\begin{Theorem}\label{Th1.1}
Assume that $u\in C^4(\Omega)\cap C^{3}(\bar \Omega)$ is an elliptic solution of the problem \eqref{1.1}-\eqref{1.2} in a $C^{3,1}$ uniformly $A$-convex domain $\Omega\subset\mathbb{M}^n$, where $A\in C^2(\bar \Omega\times \mathbb{R}\times T_x M^n)$ is regular and non-decreasing, $B>0,\in C^2(\bar \Omega\times \mathbb{R}\times T_x M^n)$ and $\varphi\in C^{2,1}(\partial \Omega\times \mathbb{R})$ are both non-decreasing. Suppose that there exists an elliptic supersolution $\bar u\in C^2(\bar \Omega)$ satisfying \eqref{super equation}-\eqref{super boundary}.
Then we have the estimate
\begin{equation}\label{global C2 bound}
\sup\limits_\Omega |\nabla^2 u|\le C,
\end{equation}
where $C$ is a constant depending on $n$, $A$, $B$, $\Omega$, $\bar u$, $\varphi$, $\delta_0$ and $|u|_{1;\Omega}$.
\end{Theorem}

Due to Theorem \ref{Th1.1}, we obtain the classical existence theorem for $(\ref{1.1})$ and $(\ref{1.2})$ under further hypotheses for the solution bounds and the gradient estimates. For the solution estimates, we can assume the existence of bounded subsolutions and supersolutions by virtue of the comparison principle. Under a further structural assumption on the matrix $A$,
\begin{equation}\label{QS}
A(x,z,p)\ge - \mu_0 [1+ |p|^2],
\end{equation}
for all $x\in \Omega$, $|z|\le K$, $p\in T_x M^n$ and some positive constant $\mu_0$ depending on the constant $K$,
we can control the gradient of elliptic solution which has been proved in \cite{xg} and extends the results for the case of Euclidean space $R^n$ in \cite{JXX}.

Combining the second derivative bounds with the lower order estimates, we can get the global second derivative H\"{o}lder estimates as in \cite{LieTru1986,LT1986,LTU1986,Tru1984} and establish the existence result by the method of continuity.

\begin{Theorem}\label{Th1.2}
Under the assumptions of Theorem \ref{Th1.1}, suppose that either $A$, $B$ or $\varphi$ is strictly increasing.  Assume that condition (\ref{QS}) holds and there is an elliptic subsolution satisfying $(\ref{sub equation})$ and $(\ref{sub boundary})$. Then the Neumann boundary value problem
\eqref{1.1}-\eqref{1.2} has a unique elliptic solution $u\in
C^{3,\alpha}(\bar \Omega)$ for any $\alpha<1$.
\end{Theorem}

The uniqueness of the solution follows from the comparison principle for the elliptic solution.
The regularity for the solution $u$ in Theorem \ref{Th1.2} can be improved by the linear elliptic theory \cite{GTbook} if the data are sufficiently smooth. For example, if $A$, $B$, $\varphi$ and $\partial \Omega$ are $C^\infty$, then $u\in C^\infty(\bar \Omega)$.

The paper is organized as follows.
Section \ref{Section 2} devotes to some preliminary results, such as a comparison principle for the Neumann problem \eqref{1.1}-\eqref{1.2}, the maximum modulus and the gradient estimate.
In Section \ref{Section 3}, we consider the second order derivative estimate in the interior of the domain and in a neighbourhood of the boundary successively. Then the proof of Theorem \ref{Th1.1} is given, which is crucial for this paper. In Section \ref{Section 4}, we provide the proof of Theorem \ref{Th1.2}.





\section{Preliminaries}\label{Section 2}

In this section, first we recall some formulae for commuting covariant derivatives on $M^n$. Then we study the maximum modulus and gradient bounds for elliptic solutions of the Neumann boundary value problem \eqref{1.1}-\eqref{1.2}. The maximum modulus is obtained from the assumed supersolution and subsolution by virtue of a comparison principle for the Neumann boundary value problem. The gradient bound has been established only using the ellipticity of the solution and a quadratic bound from below of the matrix $A$. In \cite{xg} we have obtained the gradient estimate for the degenerate elliptic solution of the problem \eqref{1.1}-\eqref{1.2},  by using the ellipticity of the solution and a quadratic bound from below
of the matrix A. Here we formulate the gradient estimate as a lemma without proof.

Let $(M^n, g)$ be an n-dimensional Riemannian manifold. Throughout the paper, $\nabla$ denotes the covariant differentiation on $M^n$. We choose a
local orthonormal vector field $\{e_1,\cdots, e_n\}$ adapted to the
Riemannian metric of $(M^n, g)$ with its dual coframe $\{\omega_1, \cdots,
\omega_n\}$. Then we have
\begin{equation}
\nabla u=\nabla_{j}u\omega_j,  \ \ \ \         \nabla_{e_i}\nu=\nabla_i\nu_ke_k,
\end{equation}
and
\begin{equation}
\nabla^2 u=\nabla_{ij}u\omega_i\omega_j,
\end{equation}
where
\begin{equation}
\nabla_{ij}u=\nabla_i(\nabla_j u)-(\nabla_i e_j)u.
\end{equation}
We recall that
\begin{equation}
\nabla_{ij}u=\nabla_{ji}u.
\end{equation}
From the Ricci identity, we have
\begin{equation}\label{ric}
\nabla_{ijk}u-\nabla_{jik}u=R_{lkji}\nabla_{l}u,
\end{equation}
where $R_{ijkl}$ is the component of the Riemannian curvature
tensor of $(M^n, g)$.
 The connection forms
$\{\omega_{ij}\}$ of $(M^n, g)$ are characterized by the structure
equations
$$d\omega_i=-\sum_{j}\omega_{ij}\wedge\omega_j,\ \
\omega_{ij}+\omega_{ji}=0,$$
$$d\omega_{ij}=-\sum_k\omega_{ik}\wedge
\omega_{kj}+\frac{1}{2}\sum_{k,l}R_{ijkl}\omega_k\wedge\omega_l.$$

We consider the distance function
\begin{equation}\label{disf}
d(x)=dist(x,x_0),
\end{equation}
in a small ball $B_r(x_0)=\{x\in\Omega,d(x)<r\}$. By choosing $r$ small enough we may assume $d^2(x)$ is smooth and
\begin{equation}
\{\delta_{ij}\}\leq\{\nabla_{ij}d^2\}\leq3\{\delta_{ij}\}
\end{equation}
in $B_r(x_0)$.

From the equation (\ref{1.1}), we have
\begin{equation}\label{9-8-1}
F[u]=\ln\det[(\nabla^2u-A(x,u,\nabla u))g^{-1}]=\tilde{B}(x,u,\nabla u),
\end{equation}
where $\widetilde{B}=\ln B$.
We still denote $$F^{ij}=\frac{\partial F}{\partial
w_{ij}}=w^{ij}\ \ , \ \ F^{ij,kl}=\frac{\partial^2 F}{\partial
w_{ij}\partial w_{kl}}=-w^{ik}w^{jl},$$ where
$\{w_{ij}\}\triangleq\{\nabla_{ij}u-A_{ij}\}$ denotes the augmented
Hessian matrix, and $\{w^{ij}\}$ denotes the inverse of the matrix
$\{w_{ij}\}$.

Now we consider the following linear operators of $F$
\begin{equation}\label{def-L}
L=w^{ij}(\nabla_{ij}-\nabla_{p_k}A_{ij}(.,u,\nabla u)\nabla_k),
\end{equation}
and
\begin{equation}
\mathcal{L}=L-\nabla_{p_k}\widetilde{B}\nabla_k.
\end{equation}
For convenience, we denote $\nabla_{\xi\eta}u\triangleq
\nabla_{ij}u\xi_i\eta_j$, $w_{\xi\eta}\triangleq w_{ij}\xi_i\eta_j =
\nabla_{ij}u\xi_i\eta_j-A_{ij}\xi_i\eta_j$ for any vectors $\xi$ and
$\eta$. As usual, $C$ denotes a constant depending on the known data
and may change from line to line in the context.

We begin with a comparison principle of the Neumann boundary value problem for the Monge-Amp\`ere type equation. We set
\begin{equation}\label{frak F}
\begin{array}{rll}
\mathfrak{F}[u]=\!&\!\!\det Mu - B(x,u,\nabla u), & {\rm for} \ x\in \Omega,\\
G(u)= \!&\!\! \nabla_\nu u- \varphi(x,u), &  {\rm for}\  x\in \partial \Omega.
\end{array}
\end{equation}
 Recall that $Mu=[\nabla^2u-A(x,u,\nabla u)]g^{-1}$ and a function $u$ is called an elliptic function of \eqref{1.1} if $Mu>0$. We recall the following comparison principle.

\begin{Lemma}\label{comparison}
Let $u$, $v$ be two elliptic functions of equation \eqref{1.1} satisfying
\begin{eqnarray}
\mathfrak{F}[u] \ge \mathfrak{F}[v] & x\in \Omega,\label{comparison equation}\\
G(u) \geq G(v) & x\in \partial \Omega. \label{comparison boundary}
\end{eqnarray}
 Assume that $A$ or $B$ are strictly increasing in $z$ and $G$ is strictly decreasing in $z$. Then we have
\begin{equation}\label{u le v}
u \le v, \  {\rm for} \  x \in \bar \Omega.
\end{equation}
\end{Lemma}

\begin{proof}
Set $w=u-v$. By a direct calculation, from \eqref{comparison equation}, we have
\begin{equation}\label{Lw ge 0}
\begin{array}{lll}
 0 \!&\! \le  \!&\! \displaystyle \mathfrak{F}[u]-\mathfrak{F}[v]\\
   \!&\!  =   \!&\! \displaystyle (\det Mu- \det Mv) - [B(x,u,\nabla u)-B(x,v,\nabla v)]\\
   \!&\!  =   \!&\! \displaystyle \int_0^1 \frac{d}{dt}\det[Mv + t(Mu-Mv)]dt - [B(x,u,\nabla u)-B(x,v,\nabla v)]\\
   \!&\!  =   \!&\! \displaystyle a^{ij}[\nabla_{ij}(u-v)-\nabla_{p_k}A_{ij}\nabla_k(u-v)-\nabla_zA_{ij}(u-v)] \\
   \!&\!      \!&\! \displaystyle - \nabla_{p_k}B\nabla_k(u-v)-\nabla_zB(u-v)\\
   \!&\!  =   \!&\! \displaystyle a^{ij}\nabla_{ij}w+b^k\nabla_kw+cw,
\end{array}
\end{equation}
where $a^{ij}=\int_0^1 C_t^{ij}dt$ and $C_t^{ij}$ is the cofactor of the element $[Mv + t(Mu-Mv)]$, $b^k=-(a^{ij}\nabla_{p_k}A_{ij}+\nabla_{p_k}B)$, $c=-(a^{ij}\nabla_zA_{ij}+\nabla_zB)$. From the boundary condition \eqref{comparison boundary}, we have
\begin{equation}\label{boundary le 0}
\begin{array}{lll}
 0 \!&\! \leq  \!&\! \displaystyle G(u)- G(v)\\
   \!&\!  =   \!&\! \displaystyle \nabla_\nu (u-v)-\varphi(x,u)+\varphi(x,v)\\
   \!&\!  =   \!&\! \displaystyle \nabla_\nu (u-v)-\varphi_z (x,\hat u)(u-v)\\
   \!&\!  =   \!&\! \displaystyle \nabla_\nu w-\varphi_z w,
\end{array}
\end{equation}
where $\hat u = \lambda u + (1-\lambda)v$ for some $\lambda\in (0,1)$ appearing by the mean value theorem. Note that the operator $\widetilde{L}=a_{ij}D_{ij}+ b^kD_k + c$ is linear and uniformly elliptic. Furthermore, by the monotonicity of both $A$ and $B$, we have $c\le 0$. Since $\varphi$ is strictly increasing, we have $\varphi_z >0$ on $\partial \Omega$. Then by Lemma 1.2 in \cite{Lie book}, $w\le 0$ in $\bar \Omega$, which leads to the conclusion \eqref{u le v}.

\end{proof}

From the comparison principle for the Neumann problem \eqref{1.1}-\eqref{1.2}, we infer the uniqueness of the solution of the problem \eqref{1.1}-\eqref{1.2} immediately.

Since we assume the existence of a $C^2$ supersolution $\bar u$ satisfying \eqref{super equation}-\eqref{super boundary} and a $C^2$ subsolution $\underline u$ satisfying \eqref{sub equation}-\eqref{sub boundary}, on the basis of Lemma \ref{comparison}, we already have an upper bound for the solution $u$, that is $u\le \bar u$ and a lower bound for the solution $u$, that is $u\geq \underline{u}$.

Next, we establish the gradient bound for elliptic solution in $\Omega$ satisfying the Neumann boundary condition. We omit its proof since it has been finished in our previous paper \cite{xg}.
\begin{Theorem}\label{Th2.1}
Let $\Omega$ be a compact domain in $(M^n, g)$ with smooth boundary, and $u$ be a degenerate elliptic solution of the Neumann problem \eqref{1.1}-\eqref{1.2}. Assume $A$ satisfies the structure condition
\begin{equation}\label{LQS}
A(x,u,\nabla u)\geq -\mu_0 (1+|\nabla u|^2)g,
\end{equation}
for all $x\in \overline{\Omega}$ and some positive constant $\mu_0$. Then we have the gradient estimate
\begin{equation}\label{GB}
\sup_{\overline{\Omega}}|\nabla u|\leq C,
\end{equation}
where $C$ depends on $n$, $g$, $\mu_0$, $\Omega$, $\varphi$ and $\sup_{\overline{\Omega}} |u|$.
\end{Theorem}

\vspace{2mm}


\section{Second derivative estimates}\label{Section 3}

In this section, we shall employ a delicate auxiliary function for our discussion to derive the second order derivative estimate and complete the proof of Theorem \ref{Th1.1}. Note that we only need to get an upper bound for the second derivative, since the lower bound can be derived from the ellipticity condition $\nabla^2u-A>0$. The interior bound can be similarly derived as the interior Pogorelev estimate in \cite{GL1996,LiuTru2010}. While in the neighbourhood of the boundary, the proof is specific for the Neumann boundary value problem as in \cite{LTU1986}. Through out this section, we take full advantage of the assumed supersolution $\bar u$.

For the arguments below, we assume the functions
$\varphi$, $\nu$ can be smoothly extended to $\bar \Omega\times \mathbb{R}$ and $\bar \Omega$ respectively. We also assume that near the boundary, $\nu$ is extended to be constant in the normal directions.

Before we deal with the second derivative estimate, we recall a fundamental lemma in \cite{JT2014, JTY2013}, which is also crucial to construct the
second derivative estimate.
\begin{Lemma}\label{lem1}
Suppose that $u$ is an elliptic solution of \eqref{1.1}, and $\bar{u}$ is a strict elliptic supersolution of \eqref{1.1}. If $A$ is regular, then
\begin{equation}
\mathcal{L}(e^{K(\bar{u}-u)})\geq \varepsilon\sum_iw^{ii}-C
\end{equation}
holds in $B_r(x_0)$ for some positive constant $K$ and uniform positive constant $\varepsilon$, where $C$ is a positive constance depending on
$n$, $g$, $A$, $B$, $\Omega$, $\bar u$ and $|u|_{1;\Omega}$.
\end{Lemma}

\begin{proof}
Since $\bar{u}$ is a strict elliptic supersolution, then there exists $\varepsilon>0$ such that $\bar{u}_\varepsilon=\bar{u}-\varepsilon d^2$ is still a supersolution of \eqref{1.1}, i.e.
\begin{equation}
F(\bar{u}_\varepsilon)\leq \tilde{B}(x,\bar{u}_\varepsilon,\nabla \bar{u}_\varepsilon).
\end{equation}
Let $v_\varepsilon=\bar{u}_\varepsilon-u$, then we have at $x_0$
\begin{equation}\label{gx2-6}
\aligned
L(\bar{u}-u)=Lv_\varepsilon+\varepsilon Ld^2
\geq Lv_\varepsilon+\varepsilon\sum_iw^{ii}.
\endaligned
\end{equation}
By the definition of $L$, we have
\begin{equation}\label{gx2-2}
\aligned
Lv_\varepsilon=w^{ij}(\nabla_{ij}v_\varepsilon-\nabla_{p_k}A_{ij}(x,u,\nabla u)\nabla_kv_\varepsilon).
\endaligned
\end{equation}
By the concavity of $F$, we have
\begin{equation}\label{gx2-1}
F(\bar{u}_\varepsilon)-F(u)\leq w^{ij}[\nabla_{ij}v_\varepsilon-A_{ij}(x,\bar{u}_\varepsilon,\nabla \bar{u}_\varepsilon)+A_{ij}(x,u,\nabla u)],
\end{equation}
Combining \eqref{gx2-2} and \eqref{gx2-1}, we get
\begin{equation}\label{gx2-3}
\aligned
Lv_\varepsilon\geq& F(\bar{u}_\varepsilon)-F(u)+w^{ij}[A_{ij}(x,\bar{u}_\varepsilon,\nabla \bar{u}_\varepsilon)-A_{ij}(x,u,\nabla u)\\
&-\nabla_{p_k}A_{ij}(x,u,\nabla u)\nabla_kv_\varepsilon]\\
=& F(\bar{u}_\varepsilon)-F(u)+w^{ij}[A_{ij}(x,\bar{u}_\varepsilon,\nabla \bar{u}_\varepsilon)-A_{ij}(x,u,\nabla \bar{u}_\varepsilon)\\
&+A_{ij}(x,u,\nabla \bar{u}_\varepsilon)-A_{ij}(x,u,\nabla u)-\nabla_{p_k}A_{ij}(x,u,\nabla u)\nabla_kv_\varepsilon].
\endaligned
\end{equation}
From \eqref{monotone A},
\begin{equation}\label{gx2-4}
\aligned
&w^{ij}[A_{ij}(x,\bar{u}_\varepsilon,\nabla \bar{u}_\varepsilon)-A_{ij}(x,u,\nabla \bar{u}_\varepsilon)]\\
=&w^{ij}\nabla_zA_{ij}(x,\hat{z},\nabla \bar{u}_\varepsilon)v_\varepsilon\geq0,
\endaligned
\end{equation}
where $u\leq\hat{z}\leq \bar{u}_\varepsilon$.
By the Taylor expansion, we have
\begin{equation}\label{gx2-5}
\aligned
&w^{ij}[A_{ij}(x,u,\nabla \bar{u}_\varepsilon)-A_{ij}(x,u,\nabla u)-\nabla_{p_k}A_{ij}(x,u,\nabla u)\nabla_kv_\varepsilon]\\
=&\frac{1}{2}w^{ij}\nabla_{p_kp_l}A_{ij}(x,u,p_\theta)\nabla_kv_\varepsilon\nabla_lv_\varepsilon,
\endaligned
\end{equation}
here $p_\theta=\theta\nabla \bar{u}_\varepsilon+(1-\theta)\nabla u$ with $0\leq\theta\leq 1$. Let $v=\bar{u}-u$.  Combining \eqref{gx2-6}-\eqref{gx2-5}, we have
\begin{equation}\label{gx2-7}
\aligned
Lv\geq\varepsilon\sum_iw^{ii}+\frac{1}{2}w^{ij}\nabla_{p_kp_l}A_{ij}(x,u,p_\theta)\nabla_kv\nabla_lv-C_1
\endaligned
\end{equation}
at $x_0$, where $C_1$ is positive constance depend on $B$, $|u|_{C^1}$ and $|\bar{u}|_{C^2}$. By a direct calculation, we have
\begin{equation}\label{gx2-8}
\aligned
Le^{Kv}=&Ke^{Kv}[Lv+Kw^{ij}\nabla_iv\nabla_jv]\\
\geq&Ke^{Kv}[\varepsilon\sum_iw^{ii}+\frac{1}{2}w^{ij}\nabla_{p_kp_l}A_{ij}(x,u,p_\theta)\nabla_kv\nabla_lv+Kw^{ij}\nabla_iv\nabla_jv-C_1].
\endaligned
\end{equation}
We assume $e_1=\frac{\nabla v}{|\nabla v|}$ when $\nabla v\neq0$ at $x_0$, or else we finish the proof. Since $A$ is regular by \eqref{regular}, it follows
\begin{equation}\label{gx2-9}
\aligned
&\frac{1}{2}w^{ij}\nabla_{p_kp_l}A_{ij}\nabla_kv\nabla_lv+Kw^{ij}\nabla_iv\nabla_jv\\
=&(\frac{1}{2}w^{ij}\nabla_{p_1p_1}A_{ij}+Kw^{11})|\nabla v|^2\\
\geq&\Big(\frac{1}{2}w^{11}\nabla_{p_1p_1}A_{11}+\sum_{i\neq1}w^{1j}\nabla_{p_1p_1}A_{1j}+Kw^{11}\Big)|\nabla v|^2.
\endaligned
\end{equation}
Since
\begin{equation}
|w^{1j}|\leq w^{11}w^{jj},
\end{equation}
by the cauchy inequality, we have
\begin{equation}
|w^{1j}|\leq \epsilon_0 w^{ii}+\frac{1}{\epsilon_0}w^{11},
\end{equation}
for positive constant $\epsilon_0$.
Hence,
\begin{equation}\label{gx927}
\begin{array}{ll}
Le^{Kv}
&\geq Ke^{Kv}[\varepsilon\sum_iw^{ii}-\frac{1}{2}\epsilon_0w^{ii}|\nabla_{p_1p_1}A_{1i}|
|\nabla_1 v|^2\\
&-\frac{1}{8\epsilon_0}w^{11}|\nabla_{p_1p_1}A_{1i}||\nabla v|^2+Kw^{11}|\nabla_1 v|^2-C_1].
\end{array}
\end{equation}
Furthermore choosing $\epsilon_0\leq \frac{\epsilon}{|\nabla_{p_1p_1}A_{1i}||\nabla_1 v|^2}$ and $K\geq \frac{|\nabla_{p_1p_1}A_{1i}|}{8\epsilon_0}$, we have
\begin{equation}
\begin{array}{ll}
\mathcal{L}e^{Kv}&=Le^{Kv}-\nabla_{p_k}\widetilde{B}\nabla_k e^{Kv}\\
&\geq Ke^{Kv}{\frac{\epsilon}{2}\sum_iw^{ii}-C_2}\\
&\geq \epsilon_1\sum_iw^{ii}-C_3,
\end{array}
\end{equation}
where $\epsilon_1=\frac{\epsilon}{2}Ke^{Kv}$.
\end{proof}

Define $\Omega_{\mu}=\{x\in \Omega {\big{|}} \ r(x):={\rm dist}(x,\partial\Omega)< \mu\}$, where $\mu$ is a positive constant. 
Here we also assume $\mu$ is small enough such that $d(x)$ is smooth  in $\Omega_{\mu}$. We assume that the unit inner
normal vector $\nu$ has been smoothly extended from $\partial\Omega$
 to $\overline{\Omega_{\mu}}$
, which can be simply achieved by taking
$\nu=\nabla r$ in
$\Omega_{\mu}$.

\begin{Lemma}\label{lem2}
Suppose that $u$ is an elliptic solution of \eqref{1.1}, and $\Omega$ is uniformly $A$-convex \eqref{domain A convexity}.
Then
\begin{equation}
\nabla^2_{\nu\nu}u\leq C(1+M_2)^{\frac{n-2}{n-1}}
\end{equation}
on $\partial\Omega$, where $M_2=\sup_{\Omega}|\nabla^2u|$, and $C$ is
a positive constance depending on $n$, $g$, $A$, $B$, $\Omega$, $\varphi$, $\delta_0$ and $|u|_{1;\Omega}$.
\end{Lemma}

\begin{proof}
Fixing $x_0\in\partial\Omega$, let $\{e_1,e_2,\cdots,e_{n} \}$ be a local orthonormal frame near $x_0$.
By a direct calculation, one has
\begin{equation}\label{gx2-11}
\aligned
L(\nabla_\nu u)=&w^{ij}[\nabla_{ijk}u\nu_k+2\nabla_{ki}u\nabla_j\nu_k+\nabla_{k}u\nabla_{ij}\nu_k\\
& -\nabla_{p_l}A_{ij}\nabla_{kl}u\nu_k-\nabla_{p_l}A_{ij}\nabla_{k}u\nabla_l\nu_k].
\endaligned
\end{equation}
Differentiating equation \eqref{1.1} along $\nu$, we get
\begin{equation}\label{gx2-10}
\aligned
w^{ij}\nabla_\nu w_{ij}=\nabla_\nu \tilde{B}+\nabla_z \tilde{B}\nabla_\nu u+\nabla_{p_k} \tilde{B}\nabla_{ik}u\nu_i.
\endaligned
\end{equation}
By the Ricci identity \eqref{ric}, it follows
\begin{equation}\label{gx2-12}
\aligned
\nabla_\nu w_{ij}=&\nabla_{kij}u\nu_k-\nabla_\nu A_{ij}-\nabla_z A_{ij}\nabla_\nu u-\nabla_{p_k} A_{ij}\nabla_{ik}u\nu_i\\
=&\nabla_{ijk}u\nu_k+R_{lijk}\nabla_lu\nu_k-\nabla_\nu A_{ij}-\nabla_z A_{ij}\nabla_\nu u-\nabla_{p_k} A_{ij}\nabla_{ik}u\nu_i.
\endaligned
\end{equation}
Since $w_{ij}=\nabla_{ij}u-A_{ij}$, we can obtain
\begin{equation}\label{gx2-13}
\aligned
w^{ij}\nabla_{ki}u=w^{ij}(w_{ki}+A_{ki})=\delta_{jk}+w^{ij}A_{ki}.
\endaligned
\end{equation}
Putting \eqref{gx2-10}, \eqref{gx2-12} and \eqref{gx2-13} into \eqref{gx2-11}, we have
\begin{equation}\label{gx2-14}
\aligned
L(\nabla_\nu u)\leq C(1+\sum_i w^{ii}+|\nabla^2u|),
\endaligned
\end{equation}
here $C$ depend on $n$, $g$, $A$, $B$, $\Omega$, and $|u|_{1;\Omega}$.
Consider $h=\nabla_\nu u-\varphi(x,u)$. From a similar computation, we get
\begin{equation}\label{gx2-15}
\aligned
Lh\leq C(1+\sum_i w^{ii}+|\nabla^2u|).
\endaligned
\end{equation}
From the positivity of $B$, we have
\begin{equation}
1\leq Cw^{ii}, \ \ \ \ (w_{ii})^{\frac{1}{n-1}}\leq C(w^{ii}).
\end{equation}
So
\begin{equation}
\aligned
Lh\leq C(1+M_2^{\frac{n-2}{n-1}})\sum_i w^{ii}.
\endaligned
\end{equation}
Since the domain $\Omega$ is $A$-convexity \eqref{domain A convexity},
\begin{equation}
\aligned
L\phi\geq \delta_0\sum_i w^{ii}.
\endaligned
\end{equation}
Choosing $-\phi$ as a barrier function, a standard barrier argument leads to
\begin{equation}
\aligned
\nabla_\nu h\leq C(1+M_2^{\frac{n-2}{n-1}}),
\endaligned
\end{equation}
which completes the proof of the lemma.
\end{proof}

Applying the tangential operator to the boundary condition $(\ref{1.2})$, we have the mixed tangential normal derivative estimate. Then from Lemma \ref{lem2}, the double normal derivative estimate has been bounded. Next, we shall adopt the method in \cite{LTU1986} to obtain the double tangential derivative bound on the boundary.  Consequently we achieve the second derivative estimate on the boundary.

Modifying the elliptic supersolution $\bar{u}$ by adding a perturbation function $-a\phi$, where $a$ is a small positive constant. Note that if $a$ is small enough then the function $\bar{u}-a\phi$ is still elliptic supersolution of (\ref{1.1}) and (\ref{1.2}). On $\partial\Omega$, we have \begin{equation}\label{9-28}\nabla_{\nu}(\bar{u}-a\phi-u)\geq a,\end{equation} by the condition (\ref{monotone varphi}).

We now consider an auxiliary function $V(x,\xi)$ given by
\begin{equation}
V(x,\xi)=e^{\alpha|\nabla (u-\lambda\phi)|^2+\beta\Phi}(w_{\xi\xi}-V'(x,\xi))
\end{equation}
for $x\in\bar{\Omega}$, $\xi\in T_xM$ with $|\xi|=1$, $$\lambda=\max_{\overline{\Omega}}|\nabla u|,$$ and
\begin{equation}
V'(x,\xi)=2g(\xi,\nu)[\nabla_{\xi'}\varphi(x,u)-g(\nabla u,\xi')-A_{\nu\xi'}],
\end{equation}
where $\xi'=\xi-g(\xi,\nu)\nu$, $\nu$ denote the extension of the inner normal vector field on $M$ and $\Phi=e^{K(\bar{u}-u-a\phi)}$.
We assume that $V$ attain its maximum at $(x_0,\xi)$.

\vskip10pt
\textbf{Case 1.}
$x_0$ is an interior point.
$\xi$ still denotes the extension of $\xi$ in a small neighborhood of $x_0$ with $\nabla\xi(x_0)=0$, and
let $\{e_1,e_2,\cdots,e_{n} \}$ be a local orthonormal frame in the neighborhood with $w_{ij}$ diagonal at $x_0$ and $w_{11}$ is the largest eigenvalue.
Set $H=\ln V$, then we have at $x_0$,
\begin{eqnarray}
\label{9-10-1}0&=&\nabla_i H=\frac{\nabla_i(w_{\xi\xi}-V')}{w_{\xi\xi}-V'} + 2\alpha \nabla_k(u-\lambda\phi) \nabla_{ik}(u-\lambda\phi) + \beta \nabla_i \Phi, \quad {\rm for}\  i=1\cdots n,\\
\label{9-9-1}0&\geq&\mathcal{L} H=\mathcal{L}\ln(w_{\xi\xi}-V')+2\alpha\mathcal{L}|\nabla (u-\lambda\phi)|^2+\beta\mathcal{L}\Phi.
\end{eqnarray}
By a direct calculation, we have
\begin{equation}\label{09}
\aligned
\mathcal{L}\ln(w_{\xi\xi}-V')=&\frac{\mathcal{L}(w_{\xi\xi}-V')}{w_{\xi\xi}-V'}-\frac{w^{ij}\nabla_i(w_{\xi\xi}-V')\nabla_j(w_{\xi\xi}-V')}{(w_{\xi\xi}-V')^2}\\
\geq&\frac{\mathcal{L}w_{\xi\xi}-\mathcal{L}V'}{w_{\xi\xi}-V'}-(1+\theta)\frac{w^{ij}\nabla_iw_{\xi\xi}\nabla_jw_{\xi\xi}}{(w_{\xi\xi}-V')^2}\\
&-C(\theta)\frac{w^{ij}\nabla_iV'\nabla_jV'}{(w_{\xi\xi}-V')^2}.
\endaligned
\end{equation}
From the definition of $\mathcal{L}$, we have
\begin{equation}\label{9-26-4}
\mathcal{L}w_{\xi\xi}=w^{ij}[\nabla_{ij}w_{\xi\xi}-\nabla_{p_k}A_{ij}\nabla_kw_{\xi\xi}]-\nabla_{p_k}\tilde{B}\nabla_kw_{\xi\xi}.
\end{equation}
Taking  derivative on both sides of the equation \eqref{9-8-1} in the direction of $\xi$, we get
\begin{equation}\label{9-9-2}
w^{ij}\nabla_\xi w_{ij}=\nabla_\xi\tilde{B}+\nabla_z\tilde{B}\nabla_\xi u+\nabla_{p_k}\tilde{B}\nabla_{jk}u\xi_j.
\end{equation}
A further differentiation in the direction of $\xi$ yields
\begin{equation}
\aligned
w^{ij}\nabla_{\xi\xi}w_{ij}=&w^{ik}w^{jl}\nabla_{\xi}w_{ij}\nabla_{\xi}w_{kl}+\nabla_{p_k}\tilde{B}\nabla_{ijk}u\xi_i\xi_j+\nabla_z\tilde{B}\nabla_{ij} u\xi_i\xi_j\\
&+\nabla_{\xi\xi}\tilde{B}+2\nabla_{\xi z}\tilde{B}\nabla_\xi u+2\nabla_{\xi p_k}\tilde{B}\nabla_{jk}u\xi_j+\nabla_{zz}\tilde{B}(\nabla_\xi u)^2\\
&+2\nabla_{zp_k}\tilde{B}\nabla_\xi u\nabla_{jk}u\xi_j+\nabla_{p_kp_l}\tilde{B}\nabla_{jk}u\nabla_{il}u\xi_i\xi_j
\endaligned
\end{equation}
at $x_0$.
Then
\begin{equation}\label{9-8-5}
\aligned
w^{ij}\nabla_{\xi\xi}w_{ij}\geq&w^{ik}w^{jl}\nabla_{\xi}w_{ij}\nabla_{\xi}w_{kl}+\nabla_{p_k}\tilde{B}\nabla_{ijk}u\xi_i\xi_j-C[1+(w_{ii})^2].
\endaligned
\end{equation}
For convenient, we define a (0,3)-tensor as follows,
\begin{equation}\label{def-C}
\aligned
T_{ijk}=&\nabla_k A_{ij}+\nabla_z A_{ij}\nabla_k u+\nabla_{p_l} A_{ij}\nabla_{kl}u\\
&-\nabla_j A_{ik}-\nabla_z A_{ik}\nabla_j u-\nabla_{p_l} A_{ik}\nabla_{jl}u.
\endaligned
\end{equation}
Besides, by the Ricci identity, we have
\begin{equation}\label{9-8-2}
\nabla_k w_{ij}-\nabla_j w_{ik}=\nabla_suR_{sijk}-T_{ijk}.
\end{equation}
By a direct computation, it follows
\begin{equation}\label{9-8-4}
\aligned
&w^{ij}[\nabla_{ij}w_{\xi\xi}-\nabla_{\xi\xi}w_{ij}]\\
=&w^{ij}[\nabla_{ij}w_{kl}-\nabla_{kl}w_{ij}]\xi_k\xi_l+2w^{ij}w_{kl}\nabla_{ij}\xi_k\xi_l,
\endaligned
\end{equation}
and from the Ricci identity and \eqref{9-8-2},
\begin{equation}\label{9-8-3}
\aligned
\nabla_{ij}w_{kl}=&\nabla_{jl}w_{ki}-\nabla_jT_{kli}+\nabla_{js}uR_{skli}+\nabla_su\nabla_jR_{skli}\\
=&\nabla_{lj}w_{ik}+w_{si}R_{sklj}+w_{sk}R_{silj}-\nabla_jT_{kli}\\
&+\nabla_{js}uR_{skli}+\nabla_su\nabla_jR_{skli}\\
=&\nabla_{lk}w_{ij}-\nabla_lT_{ikj}-\nabla_jT_{kli}+\nabla_{ls}uR_{sikj}\\
&+\nabla_su\nabla_lR_{sikj}+\nabla_{js}uR_{skli}+\nabla_su\nabla_jR_{skli}\\
&+w_{si}R_{sklj}+w_{sk}R_{silj}.
\endaligned
\end{equation}
Combining \eqref{9-8-5}, \eqref{9-8-4} and \eqref{9-8-3}, we have
\begin{equation}\label{05}
\aligned
&w^{ij}\nabla_{ij}w_{\xi\xi}-\nabla_{p_k}\tilde{B}\nabla_{k}w_{\xi\xi}\\
\geq& w^{ik}w^{jl}\nabla_{\xi}w_{ij}\nabla_{\xi}w_{kl}-w^{ij}(\nabla_lT_{ikj}+\nabla_jT_{kli})\xi_k\xi_l-C[1+\mathcal{T}w_{ii}+(w_{ii})^2].
\endaligned
\end{equation}
where $\mathcal{T}=w^{ii}$.
From the definition of the tensor $T$ given by \eqref{def-C}, it follows
\begin{equation}\label{01}
\aligned
&w^{ij}(\nabla_lT_{ikj}+\nabla_jT_{kli})\xi_k\xi_l\\
=&w^{ij}(\nabla_{p_s}A_{kl}\nabla_{jis}u-\nabla_{p_s}A_{ij}\nabla_{lks}u)\xi_k\xi_l+w^{ij}(A_{sk}R_{sijl}+A_{si}R_{skjl})\xi_k\xi_l\\
&+w^{ij}\{(\nabla_zA_{kl}\nabla_{ij}u+\nabla_{ij}A_{kl}+2\nabla_{iz}A_{kl}\nabla_ju+2\nabla_{ip_s}A_{kl}\nabla_{js}u\\
&+\nabla_{zz}A_{kl}\nabla_iu\nabla_ju+2\nabla_{p_sz}A_{kl}\nabla_{sj}u\nabla_iu+\nabla_{p_sp_t}A_{kl}\nabla_{is}u\nabla_{jt}u)\\
&-(\nabla_zA_{ij}\nabla_{kl}u+\nabla_{kl}A_{ij}+2\nabla_{kz}A_{ij}\nabla_lu+2\nabla_{kp_s}A_{ij}\nabla_{ls}u\\
&+\nabla_{zz}A_{ij}\nabla_ku\nabla_lu+2\nabla_{p_sz}A_{ij}\nabla_{sl}u\nabla_ku+\nabla_{p_sp_t}A_{ij}\nabla_{ks}u\nabla_{lt}u)\}\xi_k\xi_l.
\endaligned
\end{equation}
By a direct calculation, we have
\begin{equation}\label{02}
\aligned
w^{ij}\nabla_{p_s}A_{ij}\nabla_{lks}u\xi_k\xi_l=w^{ij}\nabla_{p_s}A_{ij}(\nabla_sw_{\xi\xi}+\nabla_sA_{\xi\xi}+\nabla_muR_{mksl}),
\endaligned
\end{equation}
and
\begin{equation}\label{03}
\aligned
&w^{ij}\nabla_{jis}u=w^{ij}\nabla_{sij}u+w^{ij}R_{misj}\nabla_mu\\
=&w^{ij}\nabla_sw_{ij}+w^{ij}(\nabla_sA_{ij}+\nabla_zA_{ij}\nabla_su+\nabla_{p_m}A_{ij}\nabla_{ms}u+R_{misj}\nabla_mu)\\
=&\nabla_s\tilde{B}+\nabla_z\tilde{B}\nabla_s u+\nabla_{p_m}\tilde{B}\nabla_{sm}u+w^{ij}(\nabla_sA_{ij}+\nabla_zA_{ij}\nabla_su\\
&+\nabla_{p_m}A_{ij}\nabla_{ms}u+R_{misj}\nabla_mu).
\endaligned
\end{equation}
So from \eqref{01}, \eqref{02} and \eqref{03}, we get
\begin{equation}\label{04}
\aligned
&w^{ij}(\nabla_lT_{ikj}+\nabla_jT_{kli})\xi_k\xi_l\\
\leq&-w^{ij}\nabla_{p_s}A_{ij}\nabla_sw_{\xi\xi}+C(\mathcal{T}+\mathcal{T}w_{ii}+1).
\endaligned
\end{equation}
Then we have by \eqref{05} and \eqref{04}
\begin{equation}\label{07}
\mathcal{L}w_{\xi\xi}\geq w^{ik}w^{jl}\nabla_{\xi}w_{ij}\nabla_{\xi}w_{kl}-C[(1+w_{ii})\mathcal{T}+(w_{ii})^2].
\end{equation}
From a similar argument, we can also have
\begin{equation}
|\mathcal{L}V'|\leq C[(1+w_{ii})\mathcal{T}+(w_{ii})^2],
\end{equation}
and
\begin{equation}\label{08}
\aligned
\frac{1}{2}\mathcal{L}|\nabla (u-\lambda\phi)|^2=&w^{ij}[\nabla_{ijk}(u-\lambda\phi)\nabla_k(u-\lambda\phi)+\nabla_{ik}(u-\lambda\phi)\nabla_{jk}(u-\lambda\phi)\\
&-\nabla_{p_s}A_{ij}\nabla_{sk}(u-\lambda\phi)\nabla_k(u-\lambda\phi)]-\nabla_{p_s}\tilde{B}\nabla_{sk}(u-\lambda\phi)\nabla_s(u-\lambda\phi).
\endaligned
\end{equation}
By a direct calculation, it follows
\begin{equation}\label{9-26-1}
\aligned
&w^{ij}\nabla_{ik}(u-\lambda\phi)\nabla_{jk}(u-\lambda\phi)\\
=&w^{ij}(w_{ik}+A_{ik}-\lambda\nabla_{ik}\phi)(w_{jk}+A_{jk}-\lambda\nabla_{jk}\phi)\\
=&w_{ii}+2A_{ii}-2\lambda\triangle\phi+w^{ij}(A_{ik}-\lambda\nabla_{ik}\phi)(A_{jk}-\lambda\nabla_{jk}\phi),
\endaligned
\end{equation}
and
\begin{equation}\label{9-26-2}
\aligned
&w^{ij}\nabla_{ijk}(u-\lambda\phi)\nabla_k(u-\lambda\phi)\\
=&w^{ij}\nabla_{ijk}u\nabla_k(u-\lambda\phi)-\lambda w^{ij}\nabla_{ijk}\phi\nabla_k(u-\lambda\phi)\\
=&w^{ij}(\nabla_{k}w_{ij}+R_{sikj}\nabla_su+\nabla_{k}A_{ij}+\nabla_{z}A_{ij}\nabla_ku+\nabla_{p_s}A_{ij}\nabla_{ks}u)\nabla_k(u-\lambda\phi)\\
&-\lambda w^{ij}\nabla_{ijk}\phi\nabla_k(u-\lambda\phi).
\endaligned
\end{equation}
Putting \eqref{9-26-1}, \eqref{9-26-2} and \eqref{9-9-2} into \eqref{08}, we have
\begin{equation}\label{9-26-3}
\aligned
\frac{1}{2}\mathcal{L}|\nabla (u-\lambda\phi)|^2\geq w_{ii}-C\mathcal{T}.
\endaligned
\end{equation}

Combining \eqref{9-9-1}, \eqref{09}, \eqref{07} and \eqref{9-26-3}, we get from Lemma \ref{lem1},
\begin{equation}\label{9-10-3}
\aligned
0\geq&\frac{w^{ik}w^{jl}\nabla_{\xi}w_{ij}\nabla_{\xi}w_{kl}}{w_{\xi\xi}-V'}-(1+\theta)\frac{w^{ij}\nabla_iw_{\xi\xi}\nabla_jw_{\xi\xi}}{(w_{\xi\xi}-V')^2}\\
&-\frac{C[(1+w_{ii})\mathcal{T}+(w_{ii})^2]}{w_{\xi\xi}-V'}-C(\theta)\frac{w^{ij}\nabla_iV'\nabla_jV'}{(w_{\xi\xi}-V')^2}\\
&+2\alpha w_{ii}+(\beta-2\alpha C)\mathcal{T}-\beta C.
\endaligned
\end{equation}
Since $w_{11}$ is the largest eigenvalue,  then from the inequality in \cite{LTU1986}, we get at $x_0$,
\begin{equation}
w^{ik}w^{jl}\nabla_{\xi}w_{ij}\nabla_{\xi}w_{kl}\geq\frac{1}{w_{11}}w^{jl}\nabla_{\xi}w_{li}\nabla_{\xi}w_{jk}\xi_{k}\xi_{i}.
\end{equation}
From \eqref{9-8-2}, we have
\begin{equation}
\nabla_{\xi}w_{jk}\xi_{k}=\nabla_{j}w_{\xi\xi}+(\nabla_suR_{skji}-T_{kji})\xi_{k}\xi_i.
\end{equation}
Then
\begin{equation}
w^{ik}w^{jl}\nabla_{\xi}w_{ij}\nabla_{\xi}w_{kl}\geq\frac{1-\theta}{w_{11}}w^{jl}\nabla_{j}w_{\xi\xi}\nabla_{l}w_{\xi\xi}-\frac{C_\theta}{w_{11}}(1+w_{ii}).
\end{equation}
Since $V'$ is bounded, one can define a quantity as follows
 $$M_1=\sup\{V'(x_0,\eta)|\eta\in T_{x_0}M,\ |\eta|=1\}.$$
For $$w_{11}\geq w_{\xi\xi},\ \ \ {\rm and}\ \ \ \ w_{\xi\xi}-V'(x_0,\xi)\geq w_{11}-V'(x_0,e_1),$$ then if
\begin{equation}\label{9-10-5}
w_{11}>\frac{M_1}{\theta},
\end{equation}
we have
\begin{equation}
|w_{11}-w_{\xi\xi}+V'(x_0,\xi)|<\theta w_{11}.
\end{equation}
We assume \eqref{9-10-5} holds, or else we get the bound for $w$, and then
\begin{equation}\label{9-10-2}
\aligned
&\frac{w^{ik}w^{jl}\nabla_{\xi}w_{ij}\nabla_{\xi}w_{kl}}{w_{\xi\xi}-V'}-(1+\theta)\frac{w^{ij}\nabla_iw_{\xi\xi}\nabla_jw_{\xi\xi}}{(w_{\xi\xi}-V')^2}\\
\geq&-\frac{3\theta}{(1-\theta)^2}\frac{w^{ij}\nabla_iw_{\xi\xi}\nabla_jw_{\xi\xi}}{w^2_{11}}-\frac{C_\theta}{w^2_{11}}(1+w_{ii}).
\endaligned
\end{equation}
By the definition of $V'$, we have
\begin{equation}\label{9-10-4}
|\nabla V'|\leq C(1+w_{ii}).
\end{equation}
Putting \eqref{9-10-2} and \eqref{9-10-4} into \eqref{9-10-3}, we get from \eqref{9-10-1} the following
\begin{equation}
\aligned
0\geq(2\alpha-C-C\alpha^2\theta )w_{ii}+(\beta-2\alpha-C-C\beta^2\theta)\mathcal{T}-\beta C.
\endaligned
\end{equation}
So we obtain the estimate $w_{ii}\leq C$ by choosing $\alpha,\beta$ large and fixing a small $\theta$.

\vskip10pt

\textbf{Case 2.} $x_0\in \partial\Omega$.
In this case, we consider the following three subcases by different directions of $\xi$.

\vspace{2mm}

{\it Subcase (i).}
$\xi=\nu$, we proved in Lemma \ref{lem2} that
\begin{equation}
\nabla_{\nu\nu}u\leq C(1+M_2)^{\frac{n-2}{n-1}}.
\end{equation}

\vspace{2mm}

{\it Subcase (ii).} $\xi$ is neither normal nor tangential to $\partial\Omega$. The unit vector $\xi$ can be written as
\begin{equation}
\xi = \xi^T+ g(\xi\cdot\nu)\nu,
\end{equation}
here $\xi^T\in T_{x_0}\partial\Omega$ is the tangential part of $\xi$. Let $$\tau=\frac{\xi^T}{|\xi^T|}.$$
Then by the constructions of $V$ and $V'$, we have
\begin{equation}
\aligned
w_{\xi\xi}=|\xi^T|^2w(\tau,\tau)+g(\xi\cdot\nu)^2w(\nu,\nu)+V'(x_0,\xi).
\endaligned
\end{equation}
So
\begin{equation}
\aligned
V(x_0,\xi)=&|\xi^T|^2V(x_0,\tau)+g(\xi\cdot\nu)^2V(x_0,\nu)\\
\leq&|\xi^T|^2V(x_0,\xi)+g(\xi\cdot\nu)^2V(x_0,\nu),
\endaligned
\end{equation}
which implies $V(x_0,\xi)\leq V(x_0,\nu)$. In fact, $V(x_0,\xi)= V(x_0,\nu)$ for $V(x_0,\xi)\geq V(x_0,\nu)$.

\vspace{2mm}

{\it Subcase (iii).} $\xi$ is tangential to $\partial \Omega$ at $x_0$.
Let $\{e_1,e_2,\cdots,e_{n}\}$ be the local orthnormal frame near $x_0$ on $\Omega$ by parallel translation of a local orthnormal frame on $\partial\Omega$
with $e_n=\nu$.
We still use $\xi$ denote the extension of $\xi$ in a small neighborhood of $x_0$ with $\nabla\xi(x_0)=0$. Then
 $$\nabla_n V\leq0 \ \ \ \ \ {\rm at}\  x_0,$$ so by (\ref{9-28}) we have
\begin{equation}\label{9-26-7}
\aligned
0\geq&(\alpha\nabla_n|\nabla (u-\lambda\phi)|^2+\beta\nabla_n\Phi)w_{\xi\xi}-\nabla_n w_{\xi\xi}-\nabla_n V'(x,\xi)\\
\geq&[2\alpha\nabla_{nn}(u-\lambda\phi)\nabla_{n}(u-\lambda\phi)+2\alpha\sum_{i=1}^{n-1}\nabla_{in}(u-\lambda\phi)\nabla_{i}(u-\lambda\phi)+\beta a]w_{\xi\xi}\\
&-\nabla_n w_{\xi\xi}-\nabla_n V'(x,\xi).
\endaligned
\end{equation}
From the boundary condition \eqref{1.2}, it follows
\begin{equation}\label{9-26-5}
\aligned
\nabla_{in}u=&\nabla_{i}\nabla_{n}u-\nabla_{\nabla_ie_n}u\\
=&\nabla_{i}\varphi(x,u)\nabla_{i}u+\nabla_{z}\varphi(x,u)(\nabla_{i}u)^2-\nabla_{\nabla_ie_n}u\nabla_{i}u,
\endaligned
\end{equation}
when $1\leq i\leq n-1$.
Since $w$ is positive definite and $\lambda\geq|\nabla u|$, then
\begin{equation}\label{9-26-6}
\aligned
\nabla_{nn}(u-\lambda\phi)\nabla_{n}(u-\lambda\phi)&=(\nabla_{n}u+\lambda)(w_{nn}+A_{nn}-\lambda\nabla_{nn}\phi)\\
&\geq(\nabla_{n}u+\lambda)(A_{nn}-\lambda\nabla_{nn}\phi)\\
&\geq-C.
\endaligned
\end{equation}
Putting \eqref{9-26-5} and \eqref{9-26-6} into \eqref{9-26-7}, we can obtain
\begin{equation}\label{9-26-8}
\aligned
0\geq&(\beta a-\alpha C)w_{\xi\xi}-\nabla_n w_{\xi\xi}-\nabla_n V'(x,\xi).
\endaligned
\end{equation}
By a direct calculation, we have
\begin{equation}
\aligned
\nabla_\nu w_{\xi\xi}=&(\nabla_{ijk}u+R_{sijk}\nabla_su)\xi_i\xi_j\nu_k-\nabla_\nu A(\xi,\xi)\\
=&\nabla_{\xi\xi}(\nabla_\nu u)-2g(\nabla_\xi \nu,\nabla_\xi\nabla u)-g(\nabla_{\xi\xi}\nu,\nabla u)\\
&+R_{sijk}\nabla_su\xi_i\xi_j\nu_k-\nabla_\nu A(\xi,\xi).
\endaligned
\end{equation}
Since $\xi$ is tangential to $\partial \Omega$ at $x_0$,
\begin{equation}
\aligned
\nabla_{\xi\xi}(\nabla_\nu u)=\nabla_z\varphi\nabla_{\xi\xi}u+\nabla_{\xi\xi}\varphi+2\nabla_{\xi z}\varphi\nabla_\xi u+\nabla_{zz}\varphi(\nabla_\xi u)^2,
\endaligned
\end{equation}
So, we have
\begin{equation}\label{9-26-10}
\aligned
\nabla_\nu w_{\xi\xi}=&\nabla_z\varphi\nabla_{\xi\xi}u+\nabla_{\xi\xi}\varphi+\nabla_{zz}\varphi(\nabla_\xi u)^2+R_{sijk}\nabla_su\xi_i\xi_j\nu_k\\
&+2\nabla_{\xi z}\varphi\nabla_\xi u-2g(\nabla_\xi \nu,\nabla_\xi\nabla u)-g(\nabla_{\xi\xi}\nu,\nabla u)-\nabla_\nu A(\xi,\xi)\\
\geq&- C[1+w_{\xi\xi}],
\endaligned
\end{equation}
and
\begin{equation}\label{9-26-9}
|\nabla_\nu V'|\leq C[1+w_{\xi\xi}].
\end{equation}
From \eqref{9-26-8}, \eqref{9-26-10} and \eqref{9-26-9}, we get
\begin{equation}
\aligned
0\geq&(\beta a-\alpha C-C)w_{\xi\xi}-C.
\endaligned
\end{equation}
Then we can finish the proof by choosing $\beta$ large enough.


\section{Proof of Theorem 1.2}\label{Section 4}
In this section, we give a brief proof of Theorem \ref{1.2}.
Since we now have the derivative estimates up to second order, we can use the continuity method to prove our existence theorem.
\begin{proof}[Proof of Theorem \ref{Th1.2}.]
By the maximum modulus in Section \ref{Section 2} together with Theorem \ref{Th2.1}, we can derive a global second derivative H\"{o}lder estimate
\begin{equation}\label{global Holder}
|u|_{2,\alpha;\Omega}\le C,
\end{equation}
for elliptic solutions $u\in C^4(\Omega)\cap C^3(\bar \Omega)$ of the semilinear Neumann boundary value problem \eqref{1.1}-\eqref{1.2} for $0<\alpha<1$. The estimate \eqref{global Holder} is obtained in \cite{LT1986}, Theorem 3.2, (see also \cite{LieTru1986,Tru1984}). With this $C^{2,\alpha}$ estimate, we can use the method of continuity, (see \cite{GTbook}, Theorem 17.22, Theorem 17.28), to derive the existence of a solution $u\in C^{2,\alpha}(\bar \Omega)$. By virtue of the maximum principles (see \cite{GTbook}, Theorem 9.1, Theorem 9.6), the proof of Theorem \ref{Th1.1} carry over to solution $u\in W^{4,n}(\Omega)\cap C^3(\bar \Omega)$. Thus, from the Schauder theory, (see \cite{GTbook}, Section 6.7), we can improve $C^{2,\alpha}(\bar \Omega)$ solutions with $0<\alpha<1$ to be in spaces $W^{4,p}(\Omega)\cap C^{3,\delta}(\bar \Omega)$ for all $p<\infty$, $0<\delta<1$. The uniqueness is from the comparison principle in Section 2, see Lemma \ref{comparison}.
\end{proof}


\vspace{2mm} {\textbf{Acknowledgements.}} The authors would like to
express their gratitude to the referees for the careful reading of
the manuscript and their comments. This work was partially supported
by the NSF of China (Grant Nos. 11501184, 11401131 and 11101132).

Correspondence: Ni Xiang.



\end{document}